\documentclass[12pt,a4paper]{article}

\usepackage[utf8]{inputenc}
\usepackage[english,russian]{babel}
\usepackage{amsfonts,amssymb,amsmath,amsthm}
\usepackage{euscript}
\usepackage{MPatr_Style_English_09.2013}

\author{Mikhail Patrakeev\footnote{Institute of Mathematics and Mechanics of UB RAS and Ural Federal University, Ekaterinburg, Russia; patrakeev@mail.ru}\footnote{Author was supported by Program of Mathematical Sciences Department, RAS (Project 12-T-1-1003) and Young researchers grant of Institute of Mathematics and Mechanics of UB RAS.}}
\title{Metrizable images of the Sorgenfrey line\\}
\date{}

\begin{document}
\hyphenation{pa-ra-com-pact}

\renewcommand{\proofname}{\textup{\textbf{Proof}}}
\renewcommand{\abstractname}{\textup{Abstract}}
\renewcommand{\refname}{\textup{References}}

\maketitle
\begin{abstract}
We give descriptions of metrizable topological spaces that are images of the Sorgenfrey line under continuous maps of different types (open, closed, quotient and others).
To obtain this descriptions, we introduce the notion of a Lusin $\pi\!$-base; the Sorgenfrey line and the Baire space have Lusin $\pi\!$-bases, and if a space $X$ has a Lusin $\pi\!$-base, then for each nonempty Polish space $Y,$ there exists a continuous open map $f:X\xrightarrow{\text{ onto }}Y.$%
\end{abstract}

\section{Introduction}
\label{subsec1}

The Sorgenfrey line is the real line with topology whose base consists of all half-open intervals  of the form $[a,b),$ where $a<b.$ We study questions of the following form:%
\medskip\\
\emph{Let $\mathcal{K}$ be a class of continuous maps and suppose $f\in\mathcal{K}$ is a map from the Sorgenfrey  line onto a metrizable space $X.$ What can we say about $X$?}%
\medskip

The answers to this questions for some classes $\mathcal{K}$ are presented in Table~1; all proofs are given in this paper (see the third column in Table~1).

\renewcommand{\arraystretch}{0}%
\begin{center}
\begin{tabular}{|c|c|c|}
\hline
\multicolumn{1}{|c|}{\it \strut $\mathcal{K},$ a class of}&\multicolumn{1}{|c|}{\it \strut\rule{0pt}{14pt}Description of metrizable images}&\multicolumn{1}{|c|}{\it Reference}\\
\multicolumn{1}{|c|}{\it \strut continuous maps}&\multicolumn{1}{|c|}{\it of the Sorgenfrey line}&\multicolumn{1}{|c|}{\it to the proof}\\
{}&\multicolumn{1}{|c|}{\it \strut\rule[-7pt]{0pt}{10pt}under maps from class $\mathcal{K}$}&\multicolumn{1}{|c|}{}
\\
\hline
\rule{0pt}{13pt}\strut All maps&{}&{}%
\\
\cline{1-1} 
\rule{0pt}{13pt}\strut Quotient maps&Nonempty analytic&{}
\\
\cline{1-1}
\rule{0pt}{13pt}\strut Pseudo-open maps&spaces&
\\
\cline{1-1}
\rule{0pt}{13pt}\strut Biquotient maps&{}&Corollary~\ref{s8.}
\\
\cline{1-2}
{}&\rule{0pt}{13pt}\strut Nonempty absolute Borel spaces&{}
\\
\strut One-to-one maps{}&\strut in which every nonempty&{}
\\
{}&\rule[-6pt]{0pt}{10pt}\strut open subset is uncountable&
\\
\hline
\rule{0pt}{13pt}\strut Open maps&\rule[-6pt]{0pt}{13pt}Nonempty Polish spaces&Corollary~\ref{s14.}
\\
\hline
\rule{0pt}{13pt}\strut Closed maps&\rule{0pt}{13pt}Nonempty at most countable&Corollary~\ref{s23.}
\\
\cline{1-1}
\rule{0pt}{12pt}\strut Closed-open maps&\rule[-6pt]{0pt}{10pt}Polish spaces&
\\
\hline
\rule{0pt}{13pt}\strut Open countable-to-one maps&{}&Corollary~\ref{s17.}
\\
\cline{1-1}
\rule{0pt}{13pt}\strut Open finite-to-one maps&There are no such spaces&{}\\
\cline{1-1}\cline{3-3}
\rule{0pt}{13pt}\strut Perfect maps&{}&Remark~\ref{z24.}
\\
\hline
\end{tabular}\\ \bigskip
Table~1
\end{center}

Some of these results were obtained earlier: the description of metrizable images of the Sorgenfrey line under all continuous maps was received by D.\,Motorov~\cite{1}; S.\,Svetlichnyi proved that every  metrizable image of the Sorgenfrey line under a continuous open map is Polish~\cite{2}; the author of this paper and N.\,Velichko independently constructed a continuous open map from the Sorgenfrey line onto the real line~\cite{patr,vel}, and N.\,Velichko proved in~\cite{vel} that for each such map, there is a point with preimage of cardinality $2^\omega$. The description of metrizable images of the Sorgenfrey line under continuous one-to-one maps can be easily derived from results of D.\,Motorov~\cite{1}.

It is interesting to note that if a class $\mathcal{K}$ is among first six classes listed in Table~1 (i.e., in case of all, quotient, pseudo-open, biquotient, one-to-one or open continuous maps), then the metrizable images of the Sorgenfrey line under maps from class $\mathcal{K}$ coincides with
the metrizable images of the Baire space under maps from class $\mathcal{K}.$ One of the reasons for this similarity is that both the Sorgenfrey line and the Baire space have Lusin \mbox{$\pi\!$-bases} (see Definition~\ref{opr4.}, Example~\ref{pri5.} and Lemma~\ref{l5.5}), and that if a space $X$ has a Lusin $\pi\!$-base, then  for each nonempty Polish space $Y,$ there exists a continuous open map $f:X\xrightarrow{\text{ onto }}Y$ (see Theorem~\ref{t9}).

\section{Notations and terminology}
\label{subsec1.5}

By $\mathbb{N}$ we denote the set of natural numbers, where $0\in\mathbb{N}.$ Let $A$ be any set and $n\in\mathbb{N};$ then by $A^n$ $(A^{<\mathbb{N}},$ $A^\mathbb{N})$ we denote the set of sequences of length $n$ from $A$ (the set of finite sequences from $A,$ the set of countably infinite sequences from $A,$ respectively).  The length of a sequence $s$ is denoted by ${\rm length}(s).$ We assume that there exists a (unique) sequence of length zero and this sequence coincides with the empty set $\varnothing;$ in particular, $A^0=\{\varnothing\}$ and ${\rm length}(\varnothing)=0.$ Suppose $s=\langle s_0,\ldots,s_{n-1}\rangle\in A^n$ and $a\in A;$ then by $s\hat{\ }a$ we denote the sequence $\langle s_0,\ldots,s_{n-1},a\rangle\in A^{n+1},$ and by ${s|m}$ we denote the sequence $\langle s_0,\ldots,s_{m-1}\rangle\in A^m,$ where $m\leq n.$  Likewise, if $x=\langle x_0,x_1,\ldots\rangle\in A^\mathbb{N},$ then the sequence $\langle x_0,\ldots,x_{n-1}\rangle$ is denoted by ${x|n};$  in particular, $x|0=\varnothing.$

The Baire space is the space $(\mathbb{N}^{\mathbb{N}},\tau_{\scriptscriptstyle \mathcal{N}}),$ where the topology  $\tau_{\scriptscriptstyle \mathcal{N}}$ is generated by the base $(\mathcal{N}_s)_{s\in\mathbb{N}^{<\mathbb{N}}}$ and $$\mathcal{N}_s:=\{x\in\mathbb{N}^\mathbb{N}:x|n=s\text{ for some }n\in\mathbb{N}\};$$
this base is called the standard base for the Baire space. Thus the Baire space is a countable infinite topological power of a countably infinite discrete space; note also that the Baire space is homeomorphic to the space of irrational numbers~\cite[Ex.\,3.4]{kech}. The set of reals is denoted by $\mathbb{R},$ a real segment and half-intervals are denoted by $[a,b],$ $[a,b),$ and $(a,b].$ The Sorgenfrey line, $\mathbf{S},$ is the space $(\mathbb{R},\tau_{\scriptscriptstyle \mathbf{S}})$ whose topology $\tau_{\scriptscriptstyle \mathbf{S}}$ is generated by the base $\big\{[a,b):a,b\in\mathbb{R}, a<b \big\}.$ A Polish space is a separable completely metrizable space.
A space $X$ is called analytic (absolute Borel) iff $X$ is homeomorphic to some analytic subset (some Borel subset) of some Polish space.

A map $f:X\rightarrow Y$  is called open (closed) iff for any open (closed) set $U\subseteq X,$ its image $f(U)$ is open (closed) in $Y;$ a map $f:X\rightarrow Y$  is called closed-open iff $f$ is closed and open. 
A surjective map $f:X\rightarrow Y$ is called quotient iff for any set $A\subseteq Y,$ the preimage $f^{-1}(A)$ is open in $X$ if and only if $A$ is open in $Y.$  A map $f:X\rightarrow Y$ is said to be finite-to-one (countable-to-one) iff for each $y\in Y,$ the preimage $f^{-1}(y)$ is finite (at most countable). The definitions of pseudo-open and biquotient maps can be found in the book~\cite{arh-pon}; we only note that in the class of continuous maps every closed map is pseudo-open, every open map is biquotient, every biquotient map is pseudo-open, and every pseudo-open map is quotient. The symbol ``$\mathsurround=0pt:=$'' means ``equals by definition''.  Other terminology can be found in the books of R.~Engelking~\cite{eng} and A.~Kechris ~\cite{kech}.

\section{Sorgenfrey line and spaces with Lusin $\pi\!$-base}
\label{subsec2}

The construction of a continuous open map from the Sorgenfrey line onto the real line~\cite{patr} uses some special family of subsets of the Sorgenfrey line. A generalization of this construction allows to build a continuous open map from  any space with analogous family of subsets onto any nonempty Polish space (Theorem~\ref{t9}). We shall call such family a Lusin \mbox{$\pi\!$-base} because this family is a Lusin scheme and a $\pi\!$-base simultaneously.

Recall that a \emph{Lusin scheme} on a set $X$ is a family $(V_s)_{s\in\mathbb{N}^{<\mathbb{N}}}$ of subsets of $X$ such that:
\begin{itemize}
  \item[(L0)] $V_s\supseteq V_{s\hat{\ }n}$ for all $s\in\mathbb{N}^{<\mathbb{N}}$ and  $n\in\mathbb{N}.$
  \item[(L1)] $V_{s\hat{\ }i}\cap V_{s\hat{\ }j}=\varnothing$ for all $s\in\mathbb{N}^{<\mathbb{N}}$ and $i\neq j$ in $\mathbb{N}.$
\end{itemize}
Consider a special case of Lusin scheme:
\begin{opr}\label{o1}
A \emph{strict Lusin scheme} on a set $X$ is a Lusin scheme $(V_s)_{s\in\mathbb{N}^{<\mathbb{N}}}$ on $X$ such that:
\begin{itemize}
  \item[\textup{(L2)}] $\mathsurround=0pt V_\varnothing =X.$
  \item[\textup{(L3)}] $\mathsurround=0pt V_s=\bigcup_{n} V_{s\hat{\ }n}$ for all $s\in\mathbb{N}^{<\mathbb{N}}.$
  \item[\textup{(L4)}] $\bigcap_{n}V_{x|n}$ is a singleton for all $x\in\mathbb{N}^\mathbb{N}.$
\end{itemize}
\end{opr}

\begin{pri}\label{pr2}
The standard base $(\mathcal{N}_s)_{s\in\mathbb{N}^{<\mathbb{N}}}$ for the Baire space $(\mathbb{N}^{\mathbb{N}},\tau_{\scriptscriptstyle \mathcal{N}})$ is a strict Lusin scheme on the set  $\mathbb{N}^\mathbb{N}.$
\end{pri}

This example is not random and the next lemma shows that every strict Lusin scheme is closely related to the Baire space:

\begin{lem}\label{l3.}
Let $(V_s)_{s\in\mathbb{N}^{<\mathbb{N}}}$ be a strict Lusin scheme on a set $X$ and let $\tau$ be the  topology on~$X$ generated by the  subbase $\{V_s:s\in\mathbb{N}^{<\mathbb{N}}\}.$ Then the space  $(X,\tau)$ is homeomorphic to the Baire space and each set $V_s$ is closed-open in $(X,\tau).$
\end{lem}

\begin{proof} It follows from Definition~\ref{o1} that for each $x\in X,$ there is a unique sequence  $\sigma(x)\in\mathbb{N}^\mathbb{N}$ such that
$$\textstyle \{x\}=\bigcap_{n}V_{\sigma(x)|n}.$$
Consider the map $\sigma:X\rightarrow\mathbb{N}^\mathbb{N}$ defined in this way. This map is a bijection, and for all $s\in\mathbb{N}^{<\mathbb{N}},$ we have $\sigma(V_s)=\mathcal{N}_s,$ where $\mathcal{N}_s$ is an element of the standard base for the Baire space $(\mathbb{N}^{\mathbb{N}},\tau_{\scriptscriptstyle \mathcal{N}}).$ It follows that the map $\sigma:(X,\tau)\rightarrow (\mathbb{N}^{\mathbb{N}},\tau_{\scriptscriptstyle \mathcal{N}})$ is a homeomorphism. Since each set $\mathcal{N}_s$ is closed-open in the Baire space, we see that every set $V_s$ is closed-open in~$(X,\tau).$ \end{proof}

\begin{opr}\label{opr4.} A \emph{Lusin $\pi\!$-base} for a space $X$ is a strict Lusin scheme  $(V_s)_{s\in\mathbb{N}^{<\mathbb{N}}}$ on $X$ such that:
\begin{itemize}
\item[\textup{(L5)}] Each set  $V_s$ is open in $X.$
\item[\textup{(L6)}] For any point $x\in X$ and any its neighbourhood $O(x),$ there are  $s\in\mathbb{N}^{<\mathbb{N}}$ and $n_0\in\mathbb{N}$ such that $x\in V_s$ and $\bigcup_{n\geq n_0}V_{s\hat{\ }n}\subseteq O(x).$
\end{itemize}%
\end{opr}%
It is clear that every Lusin $\pi\!$-base $(V_s)_{s\in\mathbb{N}^{<\mathbb{N}}}$ for a space $X$ is a $\pi\!$-base for~$X$ (i.e., every $V_s$ is nonempty open, and for any nonempty open $U\subseteq X,$ there is $V_t$ such that $V_t\subseteq U$).

\begin{pri}\label{pri5.} The standard base $(\mathcal{N}_s)_{s\in\mathbb{N}^{<\mathbb{N}}}$ for the Baire space is  a Lusin $\pi\!$-base for the Baire space.
\end{pri}

\begin{lem}\label{l5.5}
The Sorgenfrey line has a Lusin $\pi\!$-base.
\end{lem}

\begin{proof} We build a Lusin $\pi\!$-base $(V_s)_{s\in\mathbb{N}^{<\mathbb{N}}}$ for the Sorgenfrey line $(\mathbb{R},\tau_{\scriptscriptstyle \mathbf{S}})$ by recursion on ${\rm length}(s).$
Let $V_\varnothing:=\mathbb{R},$ and let the set $\{V_s:{\rm length}(s)=1\}$ be the set of all half-intervals of the form $[z,z+1),$ where $z$ is an integer. For ${\rm  length}(s)\geq 1,$ consider an interval $[a_s,b_s)=V_s$ and let $(x_n)$ be a sequence from $[a_s,b_s)$ such that $x_0:=a_s,$ $x_{m+1}>x_m,$ $x_{m+1}-x_m<1/{\rm length}(s),$ and $(x_n)$ converges to $b_s$ in the real line with Euclidean topology; then define $V_{s\hat{\ }n}:=[x_n,x_{n+1}).$ \end{proof}

It follows from Example \ref{pr2} and Lemma~\ref{l3.} that the existence of a strict Lusin scheme that is a base for topology is a characterization of the Baire space. The existence of a Lusin $\pi\!$-base is a weaker property; however, this property is sufficient to prove the next theorem:

\begin{teo}\label{t9}
Let $X$ be a space with a Lusin $\pi\!$-base and let $Y$ be a nonempty Polish space.
Then there exists a continuous open map $f:X\xrightarrow{\text{ onto }}Y.$
\end{teo}

\begin{sle}\label{s11.} Every nonempty Polish space is an image of the Sorgenfrey line under some continuous open map.\hfill$\Box$
\end{sle}

The next lemma gives a description of spaces with a Lusin $\pi\!$-base; we need this description to prove Theorem~\ref{t9}.

\begin{lem}\label{l10} For every space $X,$ conditions \textup{(A)} and \textup{(B)} are equivalent:
\begin{itemize}
  \item [\textup{(A)}]$X$ has a Lusin $\pi\!$-base.
  \item [\textup{(B)}]There is a topology $\tau$ on the set $\mathbb{N}^\mathbb{N}$  such that:
      \begin{itemize}
      \item [\textup{(B0)}]the space $(\mathbb{N}^\mathbb{N},\tau)$ is homeomorphic to $X;$
      \item [\textup{(B1)}]$\tau$ is finer than the topology $\tau_{\scriptscriptstyle \mathcal{N}}$ of the Baire space $(\mathbb{N}^{\mathbb{N}},\tau_{\scriptscriptstyle \mathcal{N}});$
      \item [\textup{(B2)}]the standard base $(\mathcal{N}_s)_{s\in\mathbb{N}^{<\mathbb{N}}}$ for the Baire space is a Lusin $\pi\!$-base for $(\mathbb{N}^\mathbb{N},\tau).$
      \end{itemize}
\end{itemize}
\end{lem}

\begin{sle}\label{s7.} There exists a continuous one-to-one map from the Sorgenfrey line onto the Baire space.\hfill$\Box$%
\end{sle}

\begin{proof}[{\bf Proof of Lemma~\ref{l10}}]
The implication $(B)\rightarrow (A)$ is trivial,  we prove  $(A)\rightarrow (B).$
Suppose $(V_s)_{s\in\mathbb{N}^{<\mathbb{N}}}$ is a Lusin $\pi\!$-base for  $X.$
Consider the map $\sigma:X\rightarrow\mathbb{N}^\mathbb{N}$ from the proof of Lemma~\ref{l3.}.
Since $\sigma$ is a bijection and $\sigma(V_s)=\mathcal{N}_s$ for all $s\in\mathbb{N}^{<\mathbb{N}},$ it follows that the topology $\tau:=\big\{\sigma(U):U\text{ is open in }X\big\}$ is required.
\end{proof}

\begin{proof}[{\bf Proof of Theorem~\ref{t9}}] Let $X$ be a space with a Lusin $\pi\!$-base and $Y$ a nonempty Polish space with a compatible complete metric  $d.$ We may assume that \begin{equation}\label{dxy}
d(x,y)\leq 1  \ \text{  for all }\  x, y\in Y
\end{equation}
--- see, for example,~\cite[Th.\,4.1.3]{eng}. By Lemma~\ref{l10}, there is a topology $\tau$ on the set $\mathbb{N}^\mathbb{N}$ that satisfies conditions (B0)--(B2).
From~(B0) it follows that if we construct a continuous open map $h:(\mathbb{N}^\mathbb{N},\tau)\xrightarrow{\text{ onto }}Y,$ then the theorem will be proved.

The construction of this map uses a family $(B_s)_{s\in\mathbb{N}^{<\mathbb{N}}}$ of open balls in $Y$  such that:
\begin{itemize}
  \item[(a)] $B_s\neq\varnothing$ for all $s\in\mathbb{N}^{<\mathbb{N}}.$
  \item[(b)] $B_\varnothing=Y.$
  \item[(c)] $B_s\supseteq \overline{B_{s\hat{\ }n}}$ for all $s\in\mathbb{N}^{<\mathbb{N}}$ and  $n\in\mathbb{N}.$
  \item[(d)] $B_s\subseteq\bigcup_{n\geq n_0}B_{s\hat{\ }n}$ for all $s\in\mathbb{N}^{<\mathbb{N}}$ and $n_0\in\mathbb{N}.$
  \item[(e)] ${\rm diam}(B_s)\leq 2^{-{\rm length}(s)}$ for all $s\in\mathbb{N}^{<\mathbb{N}} \setminus \{\varnothing\}.$
\end{itemize}
Here by  $\overline{B_{s\hat{\ }n}}$ and ${\rm diam}(B_s)$ we denote the closure of a set $B_{s\hat{\ }n}$ and the diameter of a set $B_s,$ respectively. Let the map $h:\mathbb{N}^\mathbb{N}\rightarrow  Y$ be such that
\begin{equation}\label{oprf}
\textstyle \{h(x)\}=\bigcap_{n}B_{x|n}.
\end{equation}
The map $h$ is well-defined by conditions (a), (c), (e) and by completeness of the metric  $d.$
First we show that $h$ is onto, continuous, and open; then we shall construct the family $(B_s)_{s\in\mathbb{N}^{<\mathbb{N}}}.$

Let us show that $h$ is surjective. For this we prove that
\begin{equation}\label{B6}
h(\mathcal{N}_s)=B_s \ \text{  for all }\ s\in\mathbb{N}^{<\mathbb{N}};
\end{equation}
then, by condition~(b), the surjectivity of $h$ is a special case of~\eqref{B6} with $s=\varnothing.$
For inclusion $h(\mathcal{N}_s)\subseteq B_s:$ suppose $x\in \mathcal{N}_s,$ i.e., $x|n=s$ for some $n\in\mathbb{N};$ then by~\eqref{oprf} we have $h(x)\in B_{x|n}=B_s.$ Now let us check $B_s\subseteq h(\mathcal{N}_s).$ Since the metric $d$ is complete, it follows from (c)--(e) that for each $y\in B_s,$ there is a sequence $m_0,m_1,\ldots$ from $\mathbb{N}$ such that $\textstyle \{y\}=\bigcap_n B_{s\hat{\ }m_0\!\hat{\ }\dots\!\hat{\ }m_n}.$ Using~\eqref{oprf}, we get
$h(s\hat{\ }m_0\!\hat{\ }m_1\!\hat{\ }\!\ldots)=y,$ whence $y\in h(\mathcal{N}_s).$

Let us show that $h$ is continuous. Consider a point $x\in\mathbb{N}^\mathbb{N}$ and a neighbourhood $U\subseteq Y$ of its image $h(x).$ From~\eqref{oprf} and~(e) it follows that there is $n\in\mathbb{N}$ such that $B_{x|n}\subseteq U.$ Since the topology $\tau$ satisfies condition~(B1) of Lemma~\ref{l10}, we see that the set $\mathcal{N}_{x|n}$ is a $\tau\!$-neighbourhood of $x$ such that, by~\eqref{B6}, $h(\mathcal{N}_{x|n})=B_{x|n}\subseteq U.$

Let us show that $h$ is open. Consider a set $V\subseteq \mathbb{N}^\mathbb{N}$ that is open in $(\mathbb{N}^\mathbb{N},\tau)$ and a point $y\in h(V).$ Suppose $x\in V$ and $y=h(x);$ then from (B2) of Lemma~\ref{l10} and (L6) of Definition~\ref{opr4.} it follows that there are $s\in\mathbb{N}^{<\mathbb{N}}$ and $n_0\in\mathbb{N}$ such that
\begin{equation}\label{lpb}
\textstyle x\in \mathcal{N}_s\quad\text{ and }\quad\bigcup_{n\geq n_0}\mathcal{N}_{s\hat{\ }n}\subseteq V.
\end{equation}
We have $y=h(x)\in h(\mathcal{N}_s),$ by~(\ref{B6}) we have $h(\mathcal{N}_s)=B_s,$ then, by~(d), we get $B_s\subseteq\bigcup_{n\geq n_0}B_{s\hat{\ }n},$ next using~(\ref{B6}) we obtain
$$\textstyle \bigcup_{n\geq n_0}B_{s\hat{\ }n}=\bigcup_{n\geq n_0}h(\mathcal{N}_{s\hat{\ }n}).$$
Further, $$\textstyle \bigcup_{n\geq n_0}h(\mathcal{N}_{s\hat{\ }n}) = h\big(\bigcup_{n\geq n_0}\mathcal{N}_{s\hat{\ }n}\big),$$ and finally, by~\eqref{lpb}, we get
$$\textstyle h\big(\bigcup_{n\geq n_0}\mathcal{N}_{s\hat{\ }n}\big)\subseteq h(V).$$
Thus we have found an open ball $B_s$ in $Y$ such that $y\in B_s\subseteq h\big(V\big).$ It follows that the map $h$ is open.

To conclude the proof it remains to build a family $(B_s)_{s\in\mathbb{N}^{<\mathbb{N}}}$ of open balls in $Y$ satisfying conditions (a)--(e); we build this family inductively. Let $B_s$ be the ball of center $y_s$ and radius $r_s.$ For $s=\varnothing$ define $r_\varnothing:=2$ and choose $y_\varnothing\in Y$ arbitrary. Now suppose that a ball $B_s\neq\varnothing$ is already constructed. First we define centers $y_{s\hat{\ }n}\in B_s$ such that the set $\{y_{s\hat{\ }n}:n\geq n_0\}$ is dense in $B_s$ for all $n_0\in\mathbb{N};$ we can find such centers since $Y$ is a separable metrizable space. Next, for $n\in\mathbb{N},$ we define
$$\textstyle r_{s\hat{\ }n} := \text{ minimum of } \Big\{\frac{r_s-d(y_s,y_{s\hat{\ }n})}{2}, \ 2^{-{\rm length}(s)-3}\Big\};$$
$r_{s\hat{\ }n}$ is positive because $y_{s\hat{\ }n}\in B_s\mathsurround=0pt$.

Let us show that conditions (a)--(e) are satisfied. Condition~(a) is trivial; condition~(b) follows from \eqref{dxy}; condition~(c) follows from the first restriction on $r_{s\hat{\ }n}$ and from the triangle
inequality; condition~(e) follows from the second restriction on $r_{s\hat{\ }n}.$ It remains to verify condition~(d).

Suppose $s\in\mathbb{N}^{<\mathbb{N}},$  $n_0\in\mathbb{N},$  and $y\in B_s;$ we must show that $y\in B_{s\hat{\ }k}$ for some $k\geq n_0.$
By construction of centers $y_{s\hat{\ }n},$ we can find $k\geq n_0$ such that
\begin{equation}\label{Nk} \textstyle
d(y,y_{s\hat{\ }k})<2^{-{\rm length}(s)-3}\quad\text{ and }\quad
d(y,y_{s\hat{\ }k})<\frac{r_s-d(y_s,y)}{4}.
\end{equation}
Recall the choice of $r_{s\hat{\ }k}.$ If $r_{s\hat{\ }k}=2^{-{\rm length}(s)-3},$ then it follows from \eqref{Nk} that $d(y,y_{s\hat{\ }k})<r_{s\hat{\ }k},$ equivalently, $y\in B_{s\hat{\ }k}.$ Now, suppose that
$r_{s\hat{\ }k}=\frac{r_s-d(y_s,y_{s\hat{\ }k})}{2};$ if we express $r_s$ from this equality and substitute it into~\eqref{Nk}, we shall get
$$d(y,y_{s\hat{\ }k})<\frac{2\cdot r_{s\hat{\ }k}+d(y_s,y_{s\hat{\ }k})-d(y_s,y)}{4}=
\frac{r_{s\hat{\ }k}+\frac{1}{2}\cdot\big(d(y_s,y_{s\hat{\ }k})-d(y_s,y)\big)}{2}.$$
Next, multiplying both sides by $2$ and carrying  $\frac{1}{2}\cdot d(y,y_{s\hat{\ }k})$ from left to right, we obtain
$${\textstyle \frac{3}{2}\cdot d(y,y_{s\hat{\ }k})<r_{s\hat{\ }k}+\frac{1}{2}\cdot\big(d(y_s,y_{s\hat{\ }k})-d(y_s,y)-d(y,y_{s\hat{\ }k})\big).}$$
Since, by the triangle inequality, the expression in brackets is negative or zero, we have $\frac{3}{2}\cdot d(y,y_{s\hat{\ }k})<r_{s\hat{\ }k},$ so
$d(y,y_{s\hat{\ }k})<r_{s\hat{\ }k}.$ This means that $y\in B_{s\hat{\ }k},$ which  completes the proof.%
\end{proof}

The next lemma is a strengthening of Corollary~\ref{s7.}; in its proof we use some ideas from paper~\cite{1} of D.\,Motorov.

\begin{lem}\label{l6.} Let  $f:(\mathbb{R},\tau_{{\scriptscriptstyle \mathbf{S}}})\rightarrow Y$ be a continuous map from the Sorgenfrey line to a space of at most countable weight. Then there exists a topology $\tau$ on the set $\mathbb{R}$ such that:%
\begin{itemize}
      \item [$\bullet\ $]The topology $\tau$ is weaker than the topology $\tau_{{\scriptscriptstyle \mathbf{S}}}$ of the Sorgenfrey line.
      \item [$\bullet\ $]The space $(\mathbb{R},\tau)$ is homeomorphic to the Baire space.
      \item [$\bullet\ $]The map $f:(\mathbb{R},\tau)\rightarrow Y$ is continuous.
\end{itemize}
\end{lem}

\begin{proof} Let $(B_\lambda)_{\lambda\in\Lambda}$ be an at most countable base for the space $Y.$ The Sorgenfrey line is hereditarily Lindel\"{o}f~\cite[Ex.\,3.10.C(a)]{eng}, therefore for any set $f^{-1}(B_\lambda),$ which is open in $(\mathbb{R},\tau_{{\scriptscriptstyle \mathbf{S}}}),$ there exists a sequence $\big([c_{\lambda,n},d_{\lambda,n})\big)_{n\in\mathbb{N}}$ from the base
$\big\{[c,d):c,d\in\mathbb{R}, c<d\big\}$ such that
\begin{equation}\label{zzz}
f^{-1}(B_\lambda)=\bigcup_{n\in\mathbb{N}}[c_{\lambda,n},d_{\lambda,n}).
\end{equation}
Let us build a Lusin $\pi\!$-base $(V_s)_{s\in\mathbb{N}^{<\mathbb{N}}}$ for the Sorgenfrey line in the same way as we built it in the proof of Lemma~\ref{l5.5}, where for each $s\neq\varnothing,$ we had $V_s=[a_s,b_s);$ but in addition we demand the following:
\begin{equation}\label{pibs}
\{c_{\lambda,n}:\lambda\in\Lambda,n\in\mathbb{N}\}\cup \{d_{\lambda,n}:\lambda\in\Lambda,n\in\mathbb{N}\}\subseteq
\big\{b_s:s\in\mathbb{N}^{<\mathbb{N}}\setminus\{\varnothing\}\big\}.
\end{equation}

Let $\tau$ be the topology on $\mathbb{R}$ generated by the subbase
$\big\{V_s:s\in\mathbb{N}^{<\mathbb{N}}\big\};$
this topology is weaker than $\tau_{{\scriptscriptstyle \mathbf{S}}}.$ By Lemma~\ref{l3.}, $(\mathbb{R},\tau)$ is homeomorphic to the Baire space and each set $V_s=[a_s,b_s)$ is closed-open in $(\mathbb{R},\tau).$ This implies that each set $\{x\in\mathbb{R}:x<b_s\}$ is also closed-open in $(\mathbb{R},\tau),$ whence using (\ref{pibs}) we see that every set $[c_{\lambda,n},d_{\lambda,n})$ is open in $(\mathbb{R},\tau).$
It now follows from (\ref{zzz}) that the set $f^{-1}(B_\lambda)$ is open in $(\mathbb{R},\tau)$ for all $\lambda\in\Lambda,$ hence the map $f:(\mathbb{R},\tau) \rightarrow Y$ is continuous.\end{proof}

\begin{sle}\label{s8.}
Let $Y$ be a metrizable space. Then:
\begin{itemize}
\item[\textup{(i)}]$Y$ is an image of the Sorgenfrey line under some continuous map iff $Y$ is a nonempty analytic space.
\item[\textup{(ii)}]$Y$ is an image of the Sorgenfrey line under some continuous one-to-one map iff $Y$ is a nonempty absolute Borel space in which every nonempty open subset is uncountable.
\item[\textup{(iii)}]$Y$ is an image of the Sorgenfrey line under some continuous quotient \textup{(}biquotient, pseudo-open\textup{)} map iff $Y$ is a nonempty analytic space.
\end{itemize}
\end{sle}

\begin{proof} Suppose $Y$ is a metrizable space. Since the Sorgenfrey line is separable, its metrizable images have at most countable weight. Therefore it follows from Corollary~\ref{s7.} and Lemma~\ref{l6.} that $Y$ is an image of the Sorgenfrey line under some continuous (continuous one-to-one) map iff $Y$ is an image of the Baire space under some continuous (continuous one-to-one) map.

Part~(i) of the corollary now follows from the fact that $Y$ is a continuous image of the Baire space iff $Y$ is a nonempty analytic space~\cite[Df.\,14.1 and Th.\,7.9]{kech}. Part~(ii) follows from the analogous description of metrizable continuous one-to-one images of the Baire space. One direction of this description follows from the fact that in the class of Polish spaces a continuous one-to-one image of a Borel set is Borel~\cite[Th.\,15.1]{kech} and from the fact that one-to-one maps preserve cardinality. Another direction  was proved by W.\,Sierpinski~\cite[Footnote\,1 on p.\,447]{kur1} (his result is about the space of irrationals, but the space of irrationals is homeomorphic to the Baire space~\cite[Ex.\,3.4]{kech}).

It remains to prove the implication from right to left in part~(iii). Suppose $Y$ is a nonempty analytic space; then, as we mentioned above, $Y$ is a continuous image of the Baire space. E.~Michael and A.~Stone proved in~\cite{MS} (this result was not included in the formulation of a theorem, but was actually proved on p.\,631) that in this case $Y$ is a continuous biquotient image (and hence is a pseudo-open image and is a quotient image~\cite[Ch.\,6, Pr.\,13 and 14]{arh-pon}) of the Baire space. On the other hand Corollary~\ref{s11.} says that the Baire space is a continuous open image of the Sorgenfrey line. It can easily be verified that a composition of a continuous open map and a continuous biquotient (pseudo-open, quotient) map is again a continuous biquotient (pseudo-open, quotient, respectively) map. This implies that $Y$ is a continuous biquotient (pseudo-open, quotient) image of the Sorgenfrey line.
\end{proof}

\section{Open maps and Choquet games}
\label{subsec3}
Now in order to study open maps from the Sorgenfrey line to metrizable spaces we consider the notions of Choquet game and strong Choquet game. The \emph{Choquet game} on a nonempty space $X$ is defined as follows: two players, I and II, alternately choose nonempty open sets\medskip

I $\ \ \ \ \ \ U_0\qquad\ \ \ \ U_1\qquad\qquad\ldots$

II $\qquad\ \ \ \ \ V_0\qquad\ \ \ \ \  V_1\qquad\qquad\ldots$\medskip\\
such that $U_0\supseteq V_0\supseteq U_1\supseteq V_1 \ldots\ .$ Player~II wins the run $(U_0,V_0,\ldots)$ of Choquet game on $X$ iff $\bigcap_{n}V_n\neq\varnothing;$ otherwise player~I wins this run.
A nonempty space $X$ is called a \emph{Choquet space} iff player~II has a winning strategy in the Choquet game on $X.$
The \emph{strong Choquet game} on a nonempty space~$X$ is defined in the same way, except that the $n\!$-th move of player~I is a pair $(U_n,x_n),$ where $U_n\subseteq V_{n-1}$ is open and $x_n\in U_n,$ and the $n\!$-th move of player~II is an open $V_n\subseteq U_n$ such that $x_n\in V_n.$
A nonempty space $X$ is called a \emph{strong Choquet space} iff player II has a winning strategy in the strong Choquet game on $X.$ More precise definitions of this notions can be found in~\cite{kech}.

It is easy to verify that every space with a Lusin $\pi\!$-base is a Choquet space. On the other hand, there is a separable metrizable space with a Lusin $\pi\!$-base that is not strong Choquet~\cite{MP2}. Nevertheless, the following holds:

\begin{lem}\label{l13.} The Sorgenfrey line is a strong Choquet space.
\end{lem}

\begin{proof} We must build a winning strategy for II in the strong Choquet game on the Sorgenfrey line.
Suppose the $n\!$-th move of I is $(U_n, x_n),$ where $x_n\in U_n.$ There are  $y_n$ and $z_n$ such that $[x_n,y_n)\subseteq U_n$ and $z_n\in[x_n,y_n).$ Let us tell player~II to play $V_n:=[x_n,z_n)$ in his $n\!$-th move. We have
$$U_0\supseteq[x_0,z_0]\supseteq V_0\supseteq U_1\supseteq[x_1,z_1]\supseteq V_1\ldots,$$
therefore
$$\textstyle \bigcap_{n}V_n\supseteq\bigcap_{n}[x_{n+1},z_{n+1}]\neq\varnothing,$$
hence this strategy is winning for player~II.\end{proof}

\begin{sle}\label{s14.} Let $Y$ be a metrizable space. Then $Y$ is an image of the Sorgenfrey line under some continuous open map iff $Y$ is a nonempty Polish space.
\end{sle}

\begin{proof} The implication from right to left is proved in Corollary~\ref{s11.}. The implication from left to right follows from the facts that a continuous open image of a strong Choquet space is strong Choquet~\cite[Ex.\,8.16]{kech}, that a continuous image of a separable space is separable, and that every separable metrizable strong Choquet space is Polish~\cite[Th.\,8.17]{kech}.
\end{proof}

If a Choquet space is metrizable, then player~II has a strategy such that the set $\bigcap_{n}V_n$ is always a singleton. We shall use such strategy to show that there is no continuous open countable-to-one map from the Sorgenfrey line onto a metrizable space.

\begin{opr}\label{opr15.} The \emph{strict Choquet game} on a nonempty space~$X$ is defined in the same way as Choquet game, except that player~II wins the run $(U_0,V_0,\ldots)$ iff the set $\bigcap_{n}V_n$ is a singleton. A nonempty space $X$ is called a \emph{strict Choquet space} iff player II has a winning strategy in the strict Choquet game on $X.$
\end{opr}

\begin{teo}\label{t16.} Let $X$ be a nonempty Hausdorff Choquet space, $Y$ a strict Choquet space, and $f:X\rightarrow Y$ a continuous open map. Then at least one of the following conditions holds:
\begin{itemize}
  \item[\textup{(i)}] There exists a nonempty open set $U\subseteq X$ such that the restriction $f\upharpoonright U:U\rightarrow f(U)$ is a homeomorphism.
  \item[\textup{(ii)}] The preimage $f^{-1}(y)$ of some point $y\in Y$ has cardinality $\geq 2^{\aleph_0}.$
\end{itemize}
\end{teo}

\begin{proof} Suppose that condition~(i) does not hold; we must prove that condition~(ii) holds. To do this we shall build families $(W_s)_{s\in\{0,1\}^{<\mathbb{N}}},$ $(U_s)_{s\in\{0,1\}^{<\mathbb{N}}},$ and $(V_s)_{s\in\{0,1\}^{<\mathbb{N}}}$ of subsets of $X$ and sequences
$(\widetilde{W}_n),$ $(\widetilde{U}_n),$ and $(\widetilde{V}_n)$ of subsets of $Y$ such that the following conditions hold:
\begin{itemize}
\item[(a)] $\mathsurround=0pt W_{s}, U_{s}, V_s$ are nonempty open subsets of $X$ for all $s\in\{0,1\}^{<\mathbb{N}}.$
\item[($\mathsurround=0pt {\rm \tilde{a}}$)] $\mathsurround=0pt \widetilde{W}_n, \widetilde{U}_n, \widetilde{V}_n$ are nonempty open subsets of $Y$ for all $n\in\mathbb{N}.$
\item[(b)] $\mathsurround=0pt W_{s}\supseteq U_{s}\supseteq V_s\supseteq W_{s\hat{\ }k}\ $ for all $s\in\{0,1\}^{<\mathbb{N}}$ and $k\in\{0,1\}.$
\item[($\mathsurround=0pt {\rm \tilde{b}}$)] $\mathsurround=0pt \widetilde{W}_n\supseteq \widetilde{U}_n\supseteq \widetilde{V}_n\supseteq \widetilde{W}_{n+1}\ $ for all $n\in\mathbb{N}.$
\item[(c)] For each $\sigma\in\{0,1\}^\mathbb{N},$
the sequence $(U_{\sigma|0},V_{\sigma|0},U_{\sigma|1},V_{\sigma|1},\ldots)$ is a run of Choquet game on $X$ in which player~II plays according to some (fixed) winning strategy.
\item[($\mathsurround=0pt {\rm \tilde{c}}$)] The sequence $(\widetilde{U}_0,\widetilde{V}_0,\widetilde{U}_1,\widetilde{V}_1,\ldots)$ is a run of strict Choquet game on $Y$ in which player~II plays according to some (fixed) winning strategy.
\item[(d)] The family $(W_s)_{s\in\{0,1\}^n}$ covers the set $\widetilde{W}_n$ for all $n\in\mathbb{N},$

\medskip
where we say ``a family $(P_\lambda)_{\lambda\in \Lambda}$ \emph{covers} a set $Q$'' iff the following holds:
\begin{itemize}
\item[$\circ$] $\mathsurround=0pt P_\lambda$ is a nonempty open subset of $X$ for all $\lambda\in\Lambda;$
\item[$\circ$] $\mathsurround=0pt Q$ is a nonempty open subset of $Y;$
\item[$\circ$] the family $(P_\lambda)_{\lambda\in \Lambda}$ is disjoint;
\item[$\circ$] $\mathsurround=0pt f(P_\lambda)=Q$ for all $\lambda\in\Lambda.$
\end{itemize}
\end{itemize}

Let us show that condition~(ii) of the theorem follows from (a)--(d) and ($\mathsurround=0pt {\rm \tilde{a}}$)--($\mathsurround=0pt {\rm \tilde{c}}$).
Using ($\mathsurround=0pt {\rm \tilde{c}}$), we can define the desired point $y\in Y$ by the formula   $\{y\}=\bigcap_{n}\widetilde{V}_n.$
It follows from ($\mathsurround=0pt {\rm \tilde{b}}$) that $\bigcap_{n}\widetilde{W}_n=\bigcap_{n}\widetilde{V}_n \,(=\{y\}).$
We can use (d) to show that $f(\bigcap_{n}W_{\sigma|n})\subseteq\bigcap_{n}\widetilde{W}_n \,(=\{y\})$
for all $\sigma\in\{0,1\}^\mathbb{N};$ therefore,
$$ f^{-1}(y)\supseteq \bigcup_{\sigma\in\{0,1\}^\mathbb{N}}\Big(\bigcap_{n\in\mathbb{N}}W_{\sigma|n}\Big).$$
Condition~(b) implies that $\bigcap_{n}W_{\sigma|n}=\bigcap_{n}V_{\sigma|n},$
condition~(c) implies that every set $\bigcap_{n}V_{\sigma|n}$ is not empty, thus every set $\bigcap_{n}W_{\sigma|n}$ is not empty.
Using~(d), it is easy to prove that the family
$\textstyle(\bigcap_{n}W_{\sigma|n})_{ \sigma\in\{0,1\}^\mathbb{N}}$ is disjoint.
This means that $|f^{-1}(y)|\geq |\{0,1\}^\mathbb{N}|=2^{\aleph_0},$ that is condition~(ii) of the theorem holds.

To complete the proof, it remains to build the sets $W_s, U_s, V_s, \widetilde{W}_n, \widetilde{U}_n, \widetilde{V}_n.$
Before doing this, let us prove that for any set $Q\subseteq Y$ and any finite family $(P_0,\ldots,P_m)$ that covers $Q,$ the following four statements hold:
\begin{itemize}
  \item[(p)] There exist a set $Q'\subseteq Q$ and a family $(P'_{0},P''_{0},P'_1,P'_2,\ldots,P'_m)$ that covers $Q'$ such that $P'_{0},P''_{0}\subseteq P_0,\ $ $P'_1\subseteq P_1,$ $P'_2\subseteq P_2,$ $\ldots,$ $P'_m\subseteq P_m.$
  \item[(q)] There exist a set $Q'\subseteq Q$  and a family
$(P'_0,P''_0,P'_1,P''_1,\ldots,P'_m,P''_m)$ that covers~$Q'$ such that
$P'_0,P''_0\subseteq P_0,\ $ $\ldots,$ $P'_m,P''_m\subseteq P_m.$
  \item[(r)] For any $k\in\{0,\ldots,m\}$  and any nonempty open (in $X\!$) set $R\subseteq P_k,$ there exist a set $Q'\subseteq Q$ and a family
$(P'_0,\ldots,P'_m)$ that covers $Q'$ such that
$P'_0\subseteq P_0,$ $\ldots,$ $P'_m\subseteq P_m$ and $P'_k\subseteq R.$
  \item[(s)] For any nonempty open (in $Y\!$) set $S\subseteq Q,$
there exist a set $Q'\subseteq S$  and a family
$(P'_0,\ldots,P'_m)$ that covers $Q'$ such that
$P'_0\subseteq P_0,$ $\ldots,$ $P'_m\subseteq P_m.$
\end{itemize}

Let us check that statement~(p) holds. Since every continuous open one-to-one map is a ho\-meo\-mor\-phism and since condition~(i) of Theorem~\ref{t16.} does not hold, it follows that the nonempty open set $P_0$ contains two different points $p',p''\in P_0$ such that $f(p')=f(p'').$ Let $O_{p'}, O_{p''}\subseteq P_0$ be disjoint open neighbourhoods of  $p'$ and $p''.$
Let
$$Q':=f(O_{p'})\cap f(O_{p''}),\quad P'_{0}:=f^{-1}(Q')\cap O_{p'},\quad P''_{0}:=f^{-1}(Q')\cap O_{p''},\quad \text{and} \quad P'_i:=f^{-1}(Q')\cap P_i$$
for $i\in\{1,\ldots,m\}.$ It is easy to verify that the sets $P'_{0},P''_{0},P'_1,\ldots,P'_m,$ and $Q'$ satisfies~(p). The statements (r) and~(s) can be proved by similar arguments. To prove statement~(q) it is enough to apply statement~(p) $m+1$ times.

Now we can construct the sets $W_s, U_s, V_s, \widetilde{W}_n, \widetilde{U}_n, \widetilde{V}_n$ such that (a)--(d) and ($\mathsurround=0pt {\rm \tilde{a}}$)--($\mathsurround=0pt {\rm \tilde{c}}$) hold; we build them by recursion on $n={\rm length}(s).$
If $n=0$ (that is, $\mathsurround=0pt s=\varnothing$), let  $W_\varnothing:=X$ and $\widetilde{W}_0:=f(X).$ Note that (d) holds for $n=0,$ since $\{0,1\}^0=\{\varnothing\}.$ Fix a winning strategy for player~II in the Choquet game on $X$ and a winning strategy for player~II in the strict Choquet game on $Y.$
Suppose we have constructed $W_s$ for ${\rm length}(s)\leq n,$ $U_s$ and  $V_s$ for ${\rm length}(s)<n,$ $\widetilde{W}_k$ for $k\leq n,$ and $\widetilde{U}_k$ and $\widetilde{V}_k$ for $k<n.$ Let  $\{0,1\}^n=\{s_0,\ldots,s_m\},$ where all $s_i$ are different. First we build $U_{s_0},V_{s_0},$$\ldots,$$U_{s_m},V_{s_m},$ next $\widetilde{U}_n$ and $\widetilde{V}_n,$ and finally $W_s$ for $s\in\{0,1\}^{n+1}$ and $\widetilde{W}_{n+1}.$

Let $U_{s_0}:=W_{s_0}$ and define  $V_{s_0}$ to be the set that II plays according to his fixed winning strategy in answer to
$$(U_{s_0\!|0},V_{s_0\!|0},\quad\ldots,\quad U_{s_0|(n-1)},V_{s_0|(n-1)},U_{s_0}).$$
Apply (r) to the family $(W_{s_0},\ldots,W_{s_m}),$ which covers $\widetilde{W}_n:$ for $k=0$ and
the nonempty open $V_{s_0}\subseteq W_{s_0},$ there exist a set $\widetilde{A}_n^{(0)}\subseteq \widetilde{W}_n$ and a family $\big(A_{s_0}^{(0)},\ldots,A_{s_m}^{(0)}\big)$ that covers $\widetilde{A}_n^{(0)}$ such that
$$A_{s_0}^{(0)}\subseteq W_{s_0},\quad \ldots, \quad A_{s_m}^{(0)}\subseteq W_{s_m}\quad
\text{ and }\quad A_{s_0}^{(0)}\subseteq V_{s_0}.$$

Let $U_{s_1}:=A_{s_1}^{(0)}$ and define  $V_{s_1}$ to be the set that II plays according to his fixed winning strategy in answer to
$$(U_{s_1\!|0},V_{s_1\!|0},\quad\ldots,\quad U_{s_1\!|(n-1)},V_{s_1\!|(n-1)},U_{s_1}).$$
Apply (r) to the family $\big(A_{s_0}^{(0)},\ldots,A_{s_m}^{(0)}\big),$ which covers $\widetilde{A}_n^{(0)}:$ for $k=1$ and
the nonempty open $V_{s_1}\subseteq A_{s_1}^{(0)},$ there exist a set $\widetilde{A}_n^{(1)}\subseteq \widetilde{A}_n^{(0)}$ and a family $\big(A_{s_0}^{(1)},\ldots,A_{s_m}^{(1)}\big)$ that covers $\widetilde{A}_n^{(1)}$ such that
$$A_{s_0}^{(1)}\subseteq A_{s_0}^{(0)},\quad \ldots, \quad A_{s_m}^{(1)}\subseteq A_{s_m}^{(0)}\quad
\text{ and }\quad A_{s_1}^{(1)}\subseteq V_{s_1}.$$

Repeat this process until we get a set $\widetilde{A}_n^{(m)}\subseteq \widetilde{A}_n^{(m-1)}$ and a family $\big(A_{s_0}^{(m)},\ldots,A_{s_m}^{(m)}\big)$ that covers $\widetilde{A}_n^{(m)}$  such that
$$A_{s_0}^{(m)}\subseteq A_{s_0}^{(m-1)},\quad \ldots,\quad A_{s_m}^{(m)}\subseteq A_{s_m}^{(m-1)}\quad
\text{ and }\quad A_{s_m}^{(m)}\subseteq V_{s_m}.$$

Let $\widetilde{U}_{n}:=\widetilde{A}_n^{(m)}$ and define  $\widetilde{V}_{n}$ to be the set that II plays according to his fixed winning strategy (in the strict Choquet game on $Y$) in answer to $(\widetilde{U}_0,\widetilde{V}_0,\ldots,\widetilde{U}_{n}).$
Apply (s) to the family $\big(A_{s_0}^{(m)},\ldots,A_{s_m}^{(m)}\big),$ which covers
$\widetilde{A}_n^{(m)}:$ for the nonempty open $\widetilde{V}_{n}\subseteq \widetilde{A}_n^{(m)},$ there exist a set $\widetilde{B}_n\subseteq \widetilde{V}_{n}$ and a family $(B_{s_0},\ldots,B_{s_m})$ that covers $\widetilde{B}_n$ such that
$$B_{s_0}\subseteq A_{s_0}^{(m)},\quad \ldots, \quad B_{s_m}\subseteq A_{s_m}^{(m)}.$$
Apply (q) to the family $(B_{s_0},\ldots,B_{s_m}),$ which covers $\widetilde{B}_n:$ there exist a set $\widetilde{W}_{n+1}\subseteq \widetilde{B}_n$ and a family
$(W_{s_0\!\hat{\ }0},W_{s_0\!\hat{\ }1},\ldots,W_{s_m\!\hat{\ }0},W_{s_m\!\hat{\ }1})=(W_s)_{s\in\{0,1\}^{n+1}}$
that covers $\widetilde{W}_{n+1}$ such that
$$W_{s_0\!\hat{\ }0},W_{s_0\!\hat{\ }1}\subseteq B_{s_0},\quad\ldots,\quad
W_{s_m\!\hat{\ }0},W_{s_m\!\hat{\ }1}\subseteq B_{s_m}.$$

It is not hard to check that the constructed sets satisfy conditions~(a)--(d) and ($\mathsurround=0pt {\rm \tilde{a}}$)--($\mathsurround=0pt {\rm \tilde{c}}$). This concludes the proof.
\end{proof}

\begin{sle}\label{s17.} Let $f:\mathbf{S}\rightarrow Y$ be a continuous open map from  the Sorgenfrey line onto a metrizable space $Y.$ Then the preimage $f^{-1}(y)$ of some point $y\in Y$ has cardinality~$2^{\aleph_0}.$
\end{sle}

\begin{proof}
By Lemma~\ref{l13.}, the Sorgenfrey line is a strong Choquet space. The space $Y$ is a continuous open image of $\mathbf{S},$ thus $Y$ is also a strong Choquet space~\cite[Ex.\,8.16]{kech}; therefore both Sorgenfrey line and $Y$ are Choquet spaces.

The space $Y,$ being metrizable Choquet space, is a strict Choquet space; a winning strategy for player~II can be build as follows. Fix a winning strategy for~II in the Choquet game on $Y.$  Suppose I plays $U_n$ in his $n\!$-th move. Let $U'_n\subseteq U_n$ be any nonempty open set of diameter less than $1/n.$ To win (in the strict Choquet game) II must play the set that the winning strategy in the Choquet game tells him to choose in case I played $U'_n$ instead of $U_n.$

Now can use Theorem~\ref{t16.}. Every nonempty open subset of the Sorgenfrey line contains a copy of $\mathbf{S},$ which is not metrizable (since $\mathbf{S}$ is a separable space of uncountable weight~\cite[Ch.\,2, Pr.\,104]{arh-pon}). This implies that condition~(i) of Theorem~\ref{t16.} does not hold.
\end{proof}

\section{Closed maps and scattered spaces}
\label{subsec4}

We now turn to study closed maps from the Sorgenfrey line to metric spaces, and we shall deal with scattered spaces. Let us recall some terminology. The space $X$ is called \emph{scattered} iff every nonempty subspace of $X$ contains an isolated point. By ${\rm \bf I}(A)$ we denote the set of isolated points of a subspace~$A.$ Let $X$ be a space and $\alpha$ an ordinal. The
\emph{$\alpha\!$-th Cantor--Bendixson level} of $X,$ ${\rm \bf I}_\alpha(X),$ is defined by recursion on $\alpha:$
$${\textstyle {\rm \bf I}_\alpha(X):={\rm \bf I}\Big(X\setminus\bigcup\big\{{\rm \bf I}_\beta(X):\beta<\alpha\big\}\Big).}$$
In particular, the $0$-th Cantor--Bendixson level of $X$ is the set of isolated points of $X.$ Since the family of nonempty Cantor--Bendixson levels of $X$ is disjoint, there is the first ordinal~$\alpha$ such that ${\rm \bf I}_\alpha(X)$ is empty; this ordinal $\alpha,$ denoted by ${\rm ht}(X),$ is called the \emph{Cantor--Bendixson height} of $X.$
If a space $X$ is scattered, then the family of Cantor--Bendixson levels below ${\rm ht}(X)$ is a partition of $X,$ and for each  $x\in X,$ there is a unique $\alpha$ such that $x\in {\rm \bf I}_\alpha(X);$ we call this ordinal $\alpha$ the \emph{Cantor--Bendixson height} of $x$ in $X$ and denote by ${\rm ht}(x,X).$

\begin{lem}\label{l18.}
Let $X$ be a scattered space.
\begin{itemize}
  \item[\textup{(i)}] If $A\subseteq X$ and $x\in A,\ $ then $\ {\rm ht}(x,A)\leq {\rm ht}(x,X).$
  \item[\textup{(ii)}] If $A\subseteq X,\ $ then $\ {\rm ht}(A)\leq {\rm ht}(X).$
  \item[\textup{(iii)}] Each point $x\in X$ has a neighbourhood $O(x)$ such that
$\ {\rm ht}\big(O(x)\setminus\{x\}\big)\leq {\rm ht}(x,X).$
\end{itemize}
\end{lem}

\begin{proof} (i) We $A\subseteq X$ and prove the inequality by induction on $\alpha={\rm ht}(x,X).$ Suppose that for any $y\in A$ such that ${\rm ht}(y,X)<\alpha,$ we have $\ {\rm ht}(y,A)\leq {\rm ht}(y,X).$
Let $x\in A$ be such that ${\rm ht}(x,X)=\alpha,$ i.e., $x\in {\rm \bf I}_\alpha(X).$ This means that there is a neighbourhood $O(x)$ of $x$ such that
\begin{equation}\label{Oxx}
{\textstyle O(x)\setminus\{x\}\subseteq\bigcup\big\{{\rm \bf I}_\beta(X):\beta<\alpha\big\}.}
\end{equation}
Consider the neighbourhood $U(x):=O(x)\cap A$ of $x$ in the subspace $A.$ Let $y\in U(x)\setminus\{x\}.$ Since $U(x)\setminus\{x\}\subseteq O(x)\setminus\{x\},$ it follows from~\eqref{Oxx} that ${\rm ht}(y,X)<\alpha.$ By inductive hypothesis we get
${\rm ht}(y,A)\leq {\rm ht}(y,X),$ hence ${\rm ht}(y,A)<\alpha.$ The last inequality implies that $$U(x)\setminus\{x\}\subseteq{\textstyle \bigcup\big\{{\rm \bf I}_\beta(A):\beta<\alpha\big\}}.$$
Now we have two possibilities: either $x\in \bigcup\big\{{\rm \bf I}_\beta(A):\beta<\alpha\big\}$ or
$x$ is an isolated point in the subspace $A\setminus \bigcup\big\{{\rm \bf I}_\beta(A):\beta<\alpha\big\};$ that is, either ${\rm ht}(x,A)<\alpha$ or ${\rm ht}(x,A)=\alpha.$ Since $\alpha={\rm ht}(x,X),$ in both cases we have ${\rm ht}(x,A)\leq{\rm ht}(x,X).$

(ii) Suppose $A\subseteq X.$ Since for every $x\in X,$ ${\rm ht}(x,X)<{\rm ht}(X),$ it follows from~(i) that ${\rm ht}(x,A)<{\rm ht}(X)$ for all $x\in A.$ This means that ${\rm \bf I}_{{\rm ht}(X)}(A)$ is empty, so  ${\rm ht}(A)\leq{\rm ht}(X).$

(iii) Suppose  $x\in X,$ ${\rm ht}(x,X)=\alpha,$ and $O(x)$ is a neighbourhood of $x$ such that \eqref{Oxx} holds. Suppose  $y\in O(x)\setminus\{x\}.$ By~(i) we have
$${\rm ht}\big(y,O(x)\setminus\{x\}\big)\leq {\rm ht}(y,X),$$
and by~\eqref{Oxx} ${\rm ht}(y,X)<\alpha.$ Thus ${\rm ht}\big(y,O(x)\setminus\{x\}\big)<\alpha$
for all $y\in O(x)\setminus\{x\},$ therefore ${\rm ht}\big(O(x)\setminus\{x\}\big)\leq\alpha.$ This means that the neighbourhood $O(x)$ is required.
\end{proof}

\begin{teo}\label{t19.}
Suppose $X$ is a nonempty zero-dimensional $T_1\!$-space of at most countable character such that
every nonempty closed-open subset $A\subseteq X$ can be decomposed into a countable infinite disjoint union $\bigcup_n A_n$ with each $A_n$ nonempty and closed-open in $A.$ Suppose also that $Y$
is a nonempty scattered metrizable space of at most countable cardinality. Then there exists a continuous closed-open map $f:X\xrightarrow{\text{onto}}Y.$
\end{teo}

\begin{sle}\label{s20.} Let $Y$ be a nonempty Polish space of at most countable cardinality. Then there exists a continuous closed-open map from the Sorgenfrey onto $Y.$
\end{sle}

\begin{proof}[{\bf Proof of Corollary~\ref{s20.}}] The Sorgenfrey line $\mathbf{S}$ is a nonempty zero-dimensional $T_1\!$-space of at most countable character. If $A\subseteq\mathbf{S}$ is a nonempty closed-open set, then there exists a (nonempty closed-open in $\mathbf{S}\!$) half-interval $[a,b)\subseteq A,$ and we can decompose $A$ as follows:
$$\textstyle A=\big(A\setminus[a,b)\big)\cup\bigcup_{n}\big(A\cap[b-\tfrac{b-a}{n+1},b-\tfrac{b-a}{n+2})\big).$$
Since every Polish space of at most countable cardinality is scattered~\cite[\S~34, IV, Cor.\,4]{kur1}, we can use Theorem~\ref{t19.}.
\end{proof}

Let $\mathfrak{A}$ be the class of all spaces that satisfy the conditions imposed on $X$  in the premises of Theorem~\ref{t19.}; likewise, let $\mathfrak{B}$ be the class of all spaces that satisfy the conditions imposed on $Y.$

\begin{lem}\label{l21.}
Let $\bigoplus_{\lambda\in\Lambda}W_\lambda$ be a topological sum of spaces, where $0<|\Lambda|\leq\aleph_0,$ and suppose that for each space $X\in\mathfrak{A}$ and each $\lambda\in\Lambda,$ there exists a continuous closed-open map from $X$  onto $W_\lambda.$ Then for each $X\in\mathfrak{A},$ there exists a continuous closed-open map $f:X\xrightarrow{\text{onto}}\bigoplus_{\lambda\in\Lambda}W_\lambda.$
\end{lem}

\begin{proof}[{\bf Proof of Lemma~\ref{l21.}}] First consider the case $|\Lambda|=\aleph_0;$ let $\Lambda=\{\lambda_n:n\in\mathbb{N}\}$ and all $\lambda_n$ are different. Suppose $X$ belongs to the class $\mathfrak{A};$ then $X$ can be written as a topological sum $\bigoplus_{n\in\mathbb{N}}X_n$ of nonempty subspaces. Note that each $X_n,$ being nonempty closed-open subspace of $X,$ also lies in $\mathfrak{A}.$ So, for each $n\in\mathbb{N},$ there exists a continuous closed-open map $f_{n}:X_n\xrightarrow{\text{onto}}W_{\lambda_n}.$ It is easy to verify that the sum of maps $\bigoplus_{n}f_{n}:\bigoplus_{n}X_n\longrightarrow \bigoplus_{n}W_{\lambda_n}$ is a
continuous closed-open map from $X$ onto $\bigoplus_{\lambda\in\Lambda}W_\lambda.$

Now suppose that $0<|\Lambda|<\aleph_0;$ let $\Lambda=\{\lambda_0,\ldots,\lambda_m\}.$
If we consider the set
$\textstyle X'_m:=X\setminus(X_0\cup\ldots\cup X_{m-1}),$ which
is closed-open in $X,$ then we can write $X$ as $\textstyle X_0\oplus\ldots\oplus X_{m-1}\oplus X'_m.$
The rest of construction is similar.
\end{proof}

\begin{proof}[{\bf Proof of Theorem~\ref{t19.}}] The theorem says that for every space $X$ from the class $\mathfrak{A}$ and every space $Y$ from the class $\mathfrak{B},$ there exists a continuous closed-open map $f:X\xrightarrow{\text{onto}}Y.$ We prove this by induction on $\alpha={\rm ht}(Y).$
The inductive hypothesis says that for each $X'\in\mathfrak{A}$ and each $Y'\in\mathfrak{B}$ such that  ${\rm ht}(Y')<\alpha,$ there exists a continuous closed-open map $f':X'\rightarrow Y'.$

Suppose  $X\in\mathfrak{A},$ $Y\in\mathfrak{B}$ and ${\rm ht}(Y)=\alpha.$ Using part (iii) of Lemma~\ref{l18.}, to each $y\in Y$ assign a neighbourhood $O(y)$ such that
${\rm ht}\big(O(y)\setminus\{y\}\big)\leq{\rm ht}(y,Y).$ Since ${\rm ht}(y,Y)<{\rm ht}(Y)=\alpha,$ we have
\begin{equation}\label{Oy}
{\rm ht}\big(O(y)\setminus\{y\}\big)<\alpha.
\end{equation}
The $Y$ is a nonempty regular space of at most countable cardinality, hence $Y$ is zero-dimensional~\cite[Cor.\,6.2.8]{eng}, and for each $y\in Y,$ there is a closed-open neighbourhood  $O'(y)$ such that $y\in O'(y)\subseteq O(y).$
We may assume that if $y$ is an isolated point of $Y,$ then $O'(y)=\{y\}.$
The family $\gamma:=\big\{O'(y):y\in Y\big\}$ is at most countable cover of $Y$ and members of $\gamma$ are closed-open in $Y$. Clearly, we can construct a disjoint at most countable cover $\mu$ of $Y$ such that $\mu$ refines $\gamma$ and members of $\mu$ are nonempty and closed-open in $Y.$
Since $Y$ is nonempty and can be written as $\bigoplus\{W:W\in\mu\},$ it follows from Lemma~\ref{l21.} that to conclude the proof, it remains to build a continuous closed-open map from $Z$ onto $W$ for each $Z\in\mathfrak{A}$ and $W\in\mu.$

Suppose $Z\in\mathfrak{A}$ and $W\in\mu.$ Since $\mu$ refines $\gamma,$ there exists $y_0\in Y$ such that $W\subseteq O'(y_0).$ We consider three cases:

\emph{Case 1.}  $y_0\notin W.$ Then $W\subseteq O(y_0)\setminus\{y_0\}.$
Combining \eqref{Oy} with part~(ii) of Lemma~\ref{l18.}, we get ${\rm ht}(W)<\alpha.$
Since $W$ is a nonempty subspace of $Y,$ we have $W\in\mathfrak{B},$ therefore a continuous closed-open map $f:Z\xrightarrow{\text{onto}}W$ exists by the inductive hypothesis.

\emph{Case 2.}  $y_0\in W$ and $y_0$ is an isolated point of $W. $ Then $y_0$ is an isolated point of $Y,$ therefore $O'(y_0)=\{y_0\},$ whence $W=\{y_0\}.$ Clearly, there exists a continuous closed-open map $f:Z\xrightarrow{\text{onto}}W$ in this case.

\emph{Case 3.}  $y_0\in W$ and $y_0$ is not an isolated point of $W.$
Let $z_0\in Z.$ We shall build a sequence $(Z_n)$ of subsets of $Z$ and a sequence $(W_n)$ of subsets of $W$ such that the following holds:
\begin{itemize}
  \item[(a)] $Z_n$ is a nonempty closed-open subset of $Z$ for each $n\in\mathbb{N}.$
  \item[(${\rm \tilde{a}} \mathsurround=0pt $)] $W_n$ is a nonempty closed-open subset of $W$ for each $n\in\mathbb{N}.$
  \item[(b)] The family $(Z_n)_{n\in\mathbb{N}}$ is a partition of $Z\setminus\{z_0\}.$
  \item[(${\rm \tilde{b}} \mathsurround=0pt $)] The family $(W_n)_{n\in\mathbb{N}}$ is a partition of $W\setminus\{y_0\}.$
  \item[(c)] The family $\big\{Z\setminus(Z_0\cup\ldots\cup Z_n):n\in\mathbb{N}\big\}$ is a base for the space $Z$ at the point $z_0.$
  \item[(${\rm \tilde{c}} \mathsurround=0pt $)] The family $\big\{W\setminus(W_0\cup\ldots\cup W_n):n\in\mathbb{N}\big\}$ is a base for the space $W$ at the point $y_0.$
\end{itemize}
Next we shall build a map $f:Z\rightarrow W$ such that:
\begin{itemize}
  \item[(d)] $\mathsurround=0pt f(z_0)=y_0.$
  \item[(e)] $\mathsurround=0pt f(Z_n)=W_n$ for each $n\in\mathbb{N}.$
  \item[(f)] The restriction $f|Z_n:Z_n\rightarrow W_n$ is continuous and closed-open for each $n\in\mathbb{N}.$
\end{itemize}
It follows from (a)--(f) and ($\mathsurround=0pt {\rm \tilde{a}}$)--($\mathsurround=0pt {\rm \tilde{c}}$) that the map $f:Z\rightarrow W$ is surjective, continuous and closed-open.
So we must accomplish the construction to finish the proof.

First we build the sequence $(Z_n)$ of subsets of $Z.$ Since $Z\in \mathfrak{A}$ and the set $\{z_0\}\subseteq Z$ cannot be decomposed into a countable infinite disjoint union,
$\{z_0\}$ is not closed-open in $Z;$ that is, $z_0$ is not an isolated point of $Z.$ Since $Z$ is a $T_1\!$-space of at most countable character, we can build a strictly decreasing sequence $U_0\supsetneqq U_1\supsetneqq\ldots$ of open sets in $Z$ such that the family $\{U_n:n\in\mathbb{N}\}$ is a base for $Z$ at $z_0$ and $\bigcap_n U_n=\{z_0\}.$ The space $Z$ is zero-dimensional, therefore we may assume that all $U_n$ are closed-open in $Z;$ we may also assume that $U_0=Z.$ Now let $Z_n:=U_n\setminus U_{n+1}.$  Clearly, the sequence $(Z_n)$ satisfies (a), (b), and (c).

Since $y_0$ is not an isolated point of $W,$ and $W$ is a subspace of  $Y,$  which is metrizable and (as was mentioned above) zero-dimensional, the sequence $(W_n)$ that satisfies (${\rm \tilde{a}} \mathsurround=0pt $), (${\rm \tilde{b}} \mathsurround=0pt $), and (${\rm \tilde{c}} \mathsurround=0pt $) can be built in the same way.

Finally, we  build the map $f:Z\rightarrow W.$ By~(${\rm \tilde{b}} \mathsurround=0pt $), $W_n\subseteq W\setminus\{y_0\}$ for each $n\in\mathbb{N}.$ The choice of $y_0$ implies that $W\setminus\{y_0\}\subseteq O(y_0)\setminus\{y_0\}.$ Using~\eqref{Oy} and part~(ii) of Lemma~\ref{l18.}, we obtain ${\rm ht}(W_n)<\alpha.$ Since $Z\in\mathfrak{A}$ and  $W\subseteq Y\in\mathfrak{B},$ it follows from~(a) and (${\rm \tilde{a}} \mathsurround=0pt $) that $Z_n\in \mathfrak{A}$ and  $W_n\in\mathfrak{B}.$ Now, by the inductive hypothesis,
there exists a continuous closed-open map $f_n:Z_n\xrightarrow{\text{onto}}W_n.$

Let the map $f:Z\rightarrow W$ be given by $f(z_0):=y_0$ and $f(z):=f_n(z)$ for $z\in Z_n.$ This completes the proof, since $f$ satisfies (d)--(f).
\end{proof}

The next lemma gives a reverse inclusion for descriptions of metrizable images of the Sorgenfrey line under continuous closed and under continuous closed-open maps.

\begin{lem}\label{l22.}
Suppose a metrizable space $Y$ is an image of the Sorgenfrey line under a continuous closed map.
Then $Y$ is a Polish space of at most countable cardinality.
\end{lem}

\begin{proof}
Let $f$ be a continuous closed map from the Sorgenfrey line $\mathbf{S}$ onto $Y.$
First let us show that $Y$ is at most countable.
Since the Sorgenfrey line is a paracompact space~\cite[Ex.\,5.1.31]{eng}
and  $Y$ is a $T_1\!$-space of at most countable character,
it follows that there exists a closed subspace $H\subseteq\mathbf{S}$ such that $f(H)=Y$ and the restriction $f|H: H\rightarrow Y$ is a perfect map~\cite[Ch.\,6, Pr.\,87]{arh-pon}.
If we recall that the Sorgenfrey topology $\tau_{\scriptscriptstyle \mathbf{S}}$
is stronger than the Euclidian topology, which is metrizable, then we can show that $H$ is metrizable~\cite[Ch.\,6, Pr.\,62]{arh-pon}.
The Sorgenfrey line is hereditarily separable~\cite[Ex.\,2.1.I]{eng},
therefore its metrizable subspace $H$ has at most countable base $\mathcal{B}.$
Consider a map $U:H\rightarrow \mathcal{B}$ such that $x\in U(x)\subseteq [x,x+1)$ for all $x\in H.$ The map $U$ is one-to-one, hence $H$ is at most countable, consequently $Y$ is also at most countable.

Now let us show that $Y$ is Polish; it is enough to show that $H$ is scattered.
Indeed, being scattered and metrizable, $H$ is a $G_\delta\mathsurround=0pt $-set in its completion~\cite[\S~24, III, Cor.\,1a]{kur1}, and therefore is completely metrizable~\cite[Th.\,4.3.23]{eng}. Since $Y$ is an image of $H$ under a perfect map, $Y$ is also
completely metrizable~\cite[Th.\,4.3.26 and Th.\,3.9.10]{eng}, hence, being at most countable, is a Polish space.

It remains to show that $H$ is scattered. Conversely, suppose there exists a nonempty subspace $A\subseteq H$ with no isolated points. Let $B$ be the closure of $A$ in $H;$ $B$ has no isolated points and is closed in $\mathbf{S},$ since $H$ is closed in $\mathbf{S}.$ The Sorgenfrey line is perfectly normal~\cite[Ch.\,2, Pr.\,104]{arh-pon}, therefore there is a sequence $(U_n)$ of open sets in $\mathbf{S}$ such that $B=\bigcap_{n}U_n.$ We shall build a family $\big([a_s,b_s]\big)_{s\in\{0,1\}^{<\mathbb{N}}}$ of  nondegenerate (that is, with $a_s<b_s$) real segments such that the following conditions are satisfied:
\begin{itemize}
  \item[(a)] $[a_s,b_s]\supseteq [a_{s\hat{\ }k},b_{s\hat{\ }k}]\ $ for all $s\in\{0,1\}^{<\mathbb{N}}$ and $k\in\{0,1\}.$
  \item[(b)] $[a_{s\hat{\ }0},b_{s\hat{\ }0}]\cap [a_{s\hat{\ }1},b_{s\hat{\ }1}]=\varnothing\ $ for all $s\in\{0,1\}^{<\mathbb{N}}.$
  \item[(c)] $[a_s,b_s]\subseteq U_{{\rm length}(s)}\ $ for all $s\in\{0,1\}^{<\mathbb{N}}.$
\end{itemize}

Let us show that conditions (a)--(c) yield a contradiction, which completes the proof. Condition~(b) implies that the family $\big(\bigcap_{n}[a_{\sigma|n},b_{\sigma|n}]\big)_{\sigma\in\{0,1\}^\mathbb{N}}$ is disjoint;
(a) implies that every member $\bigcap_{n}[a_{\sigma|n},b_{\sigma|n}]$ of this family is not empty; (c) implies that $\bigcap_{n}[a_{\sigma|n},b_{\sigma|n}]\subseteq\bigcap_{n}U_n$ for each $\sigma\in\{0,1\}^{\mathbb{N}}.$
Therefore the set $\bigcap_{n}U_n$ is uncountable. On the other hand, $\bigcap_{n}U_n=B\subseteq H$ and $H$ is at most countable.

We construct the family $\big([a_s,b_s]\big)_{s\in\{0,1\}^{<\mathbb{N}}}$ by recursion on ${\rm length}(s).$ If ${\rm length}(s)=0$ (that is, $\mathsurround=0pt s=\varnothing$), choose $a_\varnothing\in B$ arbitrary.
Since $B\subseteq U_0$ and $U_0$ is open in $\mathbf{S},$ there is $b_\varnothing>a_\varnothing$ such that $[a_\varnothing,b_\varnothing]\subseteq U_0.$ Now suppose inductively that a nondegenerate segment $[a_s,b_s]$ is already constructed and that $a_s\in B.$ The point $a_s$ is not isolated in $B,$ therefore we can find two different points $a_{s\hat{\ }0}, a_{s\hat{\ }1}\in(a_s,b_s)\cap B.$ As before, for each $k\in\{0,1\},$ there is $b_{s\hat{\ }k}>a_{s\hat{\ }k}$ such that $[a_{s\hat{\ }k},b_{s\hat{\ }k}]\subseteq U_{{\rm length}(s\hat{\ }k)}.$ Clearly, we can choose $b_{s\hat{\ }0}$ and $b_{s\hat{\ }1}$ in such a way that conditions (a)--(c) are satisfied. This completes the proof. \end{proof}

The following statement is an immediate  consequence of Lemma~\ref{l22.} and Corollary~\ref{s20.}:

\begin{sle}\label{s23.}
Let $Y$ be a metrizable space. Then
$Y$ is an images of the Sorgenfrey line under some continuous closed \textup{(}closed-open\textup{)} map iff $Y$ is a nonempty at most countable Polish space.
\end{sle}

\begin{zam}\label{z24.}
Suppose $f$ is a perfect map from the Sorgenfrey line onto a space $Y.$ Then $Y$ is not metrizable.
\end{zam}

\begin{proof}
Assume the contrary. In the beginning of proof of Lemma~\ref{l22.}  we showed that a subspace $H\subseteq\mathbf{S}$ is metrizable. By the same argument we can derive a metrizability of the Sorgenfrey line, which is not the case~\cite[Ch.\,2, Pr.\,104]{arh-pon}.
\end{proof}


\end{document}